\documentclass[12pt]{amsart}
\usepackage{amscd,amsmath,amssymb,amsfonts}
\usepackage{mathtools}
\usepackage{xcolor}
\usepackage[all]{xy}
\usepackage{hyperref}
\usepackage{url}
\usepackage{stmaryrd}
\usepackage{color}

\usepackage{tikz}
\usetikzlibrary{matrix,arrows,decorations.pathmorphing}

\theoremstyle{plain}
\newtheorem{thm}{Theorem}
\newtheorem{lem}[thm]{Lemma}
\newtheorem{cor}[thm]{Corollary}
\newtheorem{prop}[thm]{Proposition}

\newtheorem{defn}[thm]{Definition}

\theoremstyle{definition}

\newtheorem{claim}[thm]{Claim}

\numberwithin{thm}{section}
\renewcommand{\tilde}{\widetilde}

\newcommand{\ml}[2]{\begin{multline}\label{#1}#2 \end{multline}}
\newcommand{\ga}[2]{\begin{gather}\label{#1}#2 \end{gather}}

\newcommand{\surj}{\twoheadrightarrow}

\newcommand{\sA}{{\mathcal A}}
\newcommand{\sB}{{\mathcal B}}

\newcommand{\sK}{{\mathcal K}}
\newcommand{\sL}{{\mathcal L}}
\newcommand{\sM}{{\mathcal M}}

\newcommand{\sO}{{\mathcal O}}

\newcommand{\C}{{\mathbb C}}

\newcommand{\F}{{\mathbb F}}

\newcommand{\N}{{\mathbb N}}

\newcommand{\Q}{{\mathbb Q}}
\newcommand{\R}{{\mathbb R}}

\newcommand{\Z}{{\mathbb Z}}

\numberwithin{equation}{section}

\begin{document}

\title[ Finite presentation]{Finite presentation of  the tame fundamental group  }
\author{H\'el\`ene Esnault, Mark Shusterman, \and Vasudevan Srinivas}
\address{Freie Universit\"at Berlin, Arnimallee 3, 14195, Berlin,  Germany}
\email{esnault@math.fu-berlin.de}
\address{Harvard University, Mathematics, 1 Oxford Street, Cambridge, MA 02138, USA}
\email{mshusterman@math.harvard.edu}
\address{   TIFR, School of Mathematics \\
Homi Bhabha Road \\
400005 Mumbai, India}
\email{ srinivas@math.tifr.res.in}
\thanks{The third author was supported during part of the preparation of the article by a J. C. Bose Fellowship of the Department of Science and Technology, India. He also acknowledges support of the Department of Atomic Energy, India under project number RTI4001.}

\begin{abstract}  

Let $p$ be a prime number, and let $k$ be an algebraically closed field of characteristic $p$.
We show that the tame fundamental group of a smooth affine curve over $k$ is a projective profinite group.
We prove that the fundamental group of a smooth projective variety over $k$ is finitely presented; more generally, the tame
fundamental group of a smooth quasi-projective variety over $k$, which admits a
good compactification, is finitely presented.

\end{abstract}

\maketitle
\section{Introduction} 

The aim of this work is to study analogs  
 in characteristic $p > 0$ of some classical properties of the  topological and \'etale fundamental groups of varieties over $\mathbb{C}$.  Our main object of study for fundamental groups is the property of being finitely presented. 

If  $X$ is a smooth connected algebraic variety  over $\mathbb{C}$, the topological space on its complex points 
 $X(\mathbb{C})$ is   homotopy equivalent to a finite CW-complex, see \cite{Mor78}.
As a result, the fundamental group $\pi_1(X(\C))$ (based at some point which is irrelevant for the discussion) admits a finite presentation.
Consequently the \'{e}tale fundamental group $\pi_1^{\text{\'{e}t}}(X)$, which by the Rieman existence theorem  is continuously isomorphic to the profinite completion 
\begin{equation}
\widehat{\pi_1(X(\C))} = \varprojlim_{\substack{H \lhd \pi_1(X(\C)), \ [\pi_1(X(\C)) : H] < \infty}} \pi_1(X(\C))/H  \notag
\end{equation} 
of $\pi_1(X(\C))$, is a finitely presented profinite group.
The existence of a finite presentation for \'{e}tale fundamental groups of smooth connected varieties over other algebraically closed fields of characteristic $0$ is deduced from this.

In this work, we consider the possibility of obtaining analogous finite presentability results for a smooth connected variety $X$ over an algebraically closed field $k$ of positive characteristic $p$.

We  work throughout with profinite groups, so all group-theoretic notions should be understood in a topological sense.
For instance, all subgroups are closed, and all homomorphism are continuous.
Our notation for a closed (respectively, open) subgroup $U$ of a profinite group $\pi$ is $U \leq_c \pi$ (respectively, $U \leq_o \pi$).
In this vein, a generating set $S$ of a profinite group $\pi$ is a subset that generates a dense subgroup, so $\pi$ is (topologically) finitely generated if it admits a finite generating set $S$ in this (topological) sense.

We  use free profinite groups of finite rank, namely profinite completions of free groups on finitely many letters.
The aforementioned finite generation of a profinite group $\pi$ can then be restated as the existence of an epimorphism
$
\varphi \colon F \to \pi
$
from a free profinite group $F$ of finite rank onto $\pi$. If we can find such a surjection $\varphi$ with $\mathrm{Ker}(\varphi)$ finitely generated as a normal subgroup of $F$, we say that $\pi$ is (topologically) finitely presented (see \cite[Introduction]{Lub01}).

On a smooth connected variety $X$ over an algebraically closed field $k$ of characteristic $p>0$, 
the choice of a geometric point $x \in X$ enables one to define the \'{e}tale fundamental group $\pi_1^{\text{\'{e}t}}(X,x)$, which is a profinite group, see \cite[Expos\'e~V]{SGA1}. 
The isomorphism class of this group does not depend on the choice of $x$, so we will omit $x$ from our notation. 
Furthermore, since we only work with \'{e}tale  fundamental groups, we shall abbreviate $\pi_1^{\text{\'{e}t}}(X,x)$ by $\pi_1(X)$.

As is well known,  if $X$ is not proper, $\pi_1(X)$ is not necessarily finitely generated. For example, for each natural number $s$, and a choice of $s$ natural numbers $m_i$ prime to $p$ and pairwise different, 
  the Artin-Schreier covers $y^p - y = z^{m_i}$ of the affine line $\mathbb{A}^1_k =\mathrm{Spec}(k[z])$  yield a surjection
$\pi_1(\mathbb{A}^1_k )\surj  \oplus_{i=1}^s \Z/p\Z$.
In order to have finite generation one has to restrict (in some sense) the ramification at points on a compactification of $X$,
so  to  restrict attention to an appropriate quotient of the fundamental group.  

We consider several quotients/completions of our profinite groups.
For a profinite group $\pi$, and a prime number $\ell$, we denote by
\begin{equation}  
\pi^{(\ell)} = \varprojlim_{\substack{U \lhd_o \pi, \ [\pi : U] \ \text{is a power of} \ \ell }} 
 \pi/U  \notag
\end{equation}
the maximal pro-$\ell$ quotient of $\pi$, which is also called the pro-$\ell$ completion of $\pi$.
Similarly, we denote by 
\begin{equation}
\pi^{(\ell')} = \varprojlim_{\substack{U \lhd_o \pi, \ 
  [\pi : U] \ \text{is not divisible by} \ \ell }} \pi/U \notag
\end{equation}
the maximal pro-prime-to-$\ell$ quotient of $\pi$, which is also called the prime-to-$\ell$ completion of $\pi$.

In this work, our main way of restricting ramification is to consider
tame fundamental groups.
For simplicity, we discuss here the case of curves, and refer to Section~\ref{sec:pprime} for the general case. Let $X\hookrightarrow \bar X$ be a smooth projective compactification of a smooth connected quasi-projective curve. 
Recall that we have fixed a geometric point $x$, thus a universal cover $\tilde X\to X$ based at $x$ which identifies $\pi_1(X,x)$ with ${\rm Aut}(\tilde X/X)$.  We drop now $x$ from the notation.  The tame fundamental group $\pi_1^t(X)$ is the quotient of $\pi_1(X)$ defined by 
\ga{}{ \pi_1^t(X)= \varprojlim_Y {\rm Aut}(Y/X) \notag}
where $\tilde X\to Y\to X$ is an intermediate cover, $Y$ is connected and $Y\to X$ is finite tame.
Tameness is a condition on each point on a normal compactification $Y$: an extension of discretely valued rings is tame if the ramification index is prime to $p$ and the residue field extension is separable (\cite[Definition 2.1.2, p.30]{GM71}).
By \cite[Theorem 10.1.6]{NSW13}, the (profinite) group $\pi_1^t(X)$ is finitely generated. We have the surjections
$$\pi_1(X)\surj \pi_1^t(X)\surj \pi_1(X)^{(p')}$$ and the prime-to-$p$-completion of $\pi_1^t(X)$ is $\pi_1(X)^{(p')}$.  Moreover, if $X$ is proper, the surjection $\pi(X)\to \pi_1^t(X)$ is an isomorphism (by definition).

In higher dimension, the Lefschetz theorem and Grothendieck's specialization theorem are used to show that $\pi^t_1(X)$ is finitely generated  (see e.g. \cite[Corollary 7.2]{Esn17},  \cite[Theorem~1.1 (a)]{EK16}).

The main theorem of our note is the following.

\begin{thm} \label{thm:main}

Let $X$ be a smooth connected quasi-projective variety defined over an algebraically closed field $k$ of characteristic $p>0$.   
If $X$ admits a smooth projective compactification $j: X\hookrightarrow \bar X$ such that $\bar X\setminus X$ is a normal crossings divisor, then $\pi_1^t(X)$ is finitely presented.
\end{thm}

If $X$ is smooth projective then  Theorem~\ref{thm:main}   says that $\pi_1(X)$ is finitely presented. 
In case $X$ is a smooth projective curve, Theorem~\ref{thm:main} 
is a consequence of \cite[Corollary 1.2]{Shu18}, which also gives a similar result over finite fields, including a determination of the number of relations in a presentation.
Obtaining (a good bound on) the number of relations in higher dimension remains a challenge.

If $X$ is a smooth connected variety defined over an algebraically closed field, we know 
by  Mich\`ele
Raynaud  (\cite[Th\'eor\`eme~2.3.1, Remarque~2.3.2]{Ray72}) under the assumption of strong desingularization
(\cite[I 3.1.5]{SGA5}) and by Orgogozo in general (\cite[Th\'eor\`eme~4.11]{Org03}) that $\pi_1(X)^{(p')}$ is finitely presented.  The proof relies on \'etale descent (\cite[Expos\'e IX, Th\'eor\`eme~4.10]{SGA1})  for $\pi_1(X)^{(p')}$  to reduce to a smooth projective normal crossings divisor compactification, 
then in this situation on the Lefschetz theorem for $\pi_1(X)^{(p')}$ (\cite[Remarque~4.3]{Ray74}) and for smooth surfaces on a Lefschetz pencil method. As  a descended finite \'etale cover of a tame finite \'etale cover is not necessarily tame, we can  not make the reduction to a good compactification case in Theorem~\ref{thm:main}. Our method also gives another proof of the finite presentation of $\pi_1(X)^{(p')}$, see Corollary~\ref{cor:primetop}.

For $X$ a curve, by \cite[Expos\'e XIII, Corollaire~2.12]{SGA1}, the group $\pi_1^{(p')}(X)$ is described by the presentation
\begin{equation}
\langle x_1, \dots, x_{2g}, y_1, \dots ,y_n : [x_1,x_2] \cdots [x_{2g-1}, x_{2g}] y_1 \cdots y_n = 1 \rangle
\notag
\end{equation}
where $g$ is the genus of $X$, $n$ is the number of points on a compactification of $X$ that are not on $X$, and $[a,b] = aba^{-1}b^{-1}$ is the commutator.
This presentation is obtained by lifting $X$ to characteristic $0$.
Theorem \ref{thm:main}  cannot be proven in this way since certain quasi-projective varieties do not lift to characteristic $0$.

Recall that a profinite group $\pi$ is called projective if for every diagram
\ga{}{\xymatrix{ &\ar@{.>}[dl]_{\exists \theta} \pi \ar[d]^{\lambda} \\
A \ar[r]_{\varphi} & B
} \notag
}
 where $\varphi$ and $\lambda$ are surjective homomorphisms of profinite groups, there exists a homomorphism $\theta \colon \pi \to A$
 making the diagram commute. A diagram as above is called an embedding problem, and the map $\theta$ is called a weak solution.
By  \cite[Proposition 7.6.7]{RZ10} we know  that $\pi$ is projective if and only if its cohomological dimension is at most $1$.
By \cite[Lemma 7.6.3]{RZ10}, projectivity is also equivalent to being isomorphic to a subgroup of a free profinite group.
For a prime number $\ell$, being a projective pro-$\ell$ group is equivalent to being a free pro-$\ell$ group by \cite[Theorem 7.7.4]{RZ10}. 
Every finitely generated projective profinite group is finitely presented by \cite[Proposition 1.1]{Lub01}.

If $X$ is a smooth curve,  the group $\pi_1^t(X)$ is finitely presented by Theorem~\ref{thm:main}.  
We strengthen this in the affine case.

\begin{thm} \label{thm:proj}

Let $X$ be a smooth connected affine curve defined over an algebraically closed field $k$ of characteristic $p$.
Then $\pi^t_1(X)$ is projective.

\end{thm}

This result is a variant of the folklore fact that $\pi_1(X)$ is projective, see for instance \cite[Theorems 10.1.2, 10.1.12]{NSW13}. 
Our proof proceeds along similar lines, one difference being the need to invoke a result of Shafarevich, saying that the maximal pro-$p$ quotient of the fundamental group of a smooth compactification $\bar X$ of $X$ is a free pro-$p$ group. 

We  enhance Theorem~\ref{thm:proj} by introducing the notion due to Drinfeld of ramification bounded by a given finite \'etale connected cover 
$h: X'\to X$. Under the assumption that $h$ is saturated (see Definition~\ref{defn:drinfeld}), there is  a naturally defined  quotient $\pi_1^h(X)$ of $\pi_1(X)$ with the property that a Galois cover $Y \to X$ has ramification bounded by $h$ precisely when the surjection $\pi_1(X) \to \mathrm{Aut}(Y/X)$ factors through $\pi_1^h(X)$.  
The quotient $\pi_1(X)\to \pi_1^t(X)$ factors through $\pi_1^h(X)$ and is equal to it if $h$ is an isomorphism. 
We prove in Theorem~\ref{thm:Drinfeld} that for a smooth connected affine curve $X$ over an algebraically closed field of positive characteristic, the group $\pi_1^h(X)$  is finitely presented.

 \medskip

We shall now discuss the methods we use. 
The starting point of our proof of Theorem~\ref{thm:main} is the following characterization of finite presentability taken from \cite[Theorem 0.3]{Lub01}.

\begin{thm}[Lubotzky] \label{thm:lub}
Let $\pi$ be a finitely generated profinite group.
Then $\pi$ is finitely presented if and only if there is a constant $C\in \R_{\ge 0}$ such that for any $r$-dimensional linear representation $M$ of $\pi$  over a prime finite field $\F$, one has 
${\rm dim}_{\F} H^2(\pi, M) \leq C \cdot r.$ \notag

\end{thm}
In \cite[Theorem.~0.2]{Lub01} there is a (more refined) cohomological expression for the number of relations of $\pi$, but we do not use it here.
In order to obtain a bound on the growth of $H^2(\pi, M)$, linear in $r$, for $\pi=\pi_1^t(X)$
 we use different methods according to the characteristic of $\F$. 
 
  If $\F$ has characteristic prime to $p$, we could rely on Lubotsky's theorem together with the finite generation of $\pi(X)^{(p')}$ to obtain one part of the necessary information. We do not do this, in particular to show how efficient is Lubotsky's theorem: it proves everything at once.  In fact, the linear bound can be shown to exist without assuming the existence of a  smooth normal crossings compactification, see Proposition~\ref{prop:tamel}.
  So assuming $\F$ has characteristic prime to $p$,
  we first compute that the natural map  
$H^2(\pi^t_1(X), M) \to H^2(X, M)$ is injective (see Corollary~\ref{cor:inj}).  Using this and an alteration \`a la de Jong-Gabber-Temkin, we can reduce  the problem to the case where $X$ has a smooth projective compactification with a normal crossings divisor at infinity. 
We can then apply the Lefschetz theorem for the tame fundamental group \cite[Theorem~1.1]{EK16}, thus reducing the problem to surfaces (varieties $X$ of dimension $2$). 

We use purity and Deligne's theorem \cite[Corollaire~2.7]{Ill81} according to which the Euler characteristic of our tame $M$  depends only on the geometry of $X$ and on $r$. This enables us to obtain the linear bound for each $\F_\ell$ separately. 
Showing that the bound can be made independent of $\ell$ relies on the cohomological representation of the zeta function in case $k$ is the algebraic closure of $\F_p$.

If $\F=\F_p$, we need numerical tameness for $\pi^t_1(X)$ (rather than just tameness), a notion introduced in \cite[Section~5]{KS10}, which says that the local inertial groups of the finite quotients of $\pi^t_1(X)$ have order prime to $p$. This condition is fulfilled when $X$ has a smooth compactification $j: X \hookrightarrow \bar X$ with a normal crossings divisor at infinity (\cite[Theorem~5.4 (a)]{KS10}). This is the reason why we need this assumption. This allows us to conclude that  Grothendieck's spectral sequence with computes equivariant cohomology  in the tower defined by the finite quotients of $\pi^t(X)$ degenerates, see \eqref{b}, yielding an inclusion
$H^2(\pi^t_1(X), M)\subset H^2(\bar X, j_*M)$, see Proposition~\ref{prop:numtame}.  Again the Lefschetz theorem {\it loc.cit.} reduces the problem to dimension $2$. 

In dimension $2$, we construct a locally free sheaf $\sM$ on $\bar X$ of rank $r$ which is acted on by the Frobenius $F$ and admits an Artin-Schreier exact sequence which describes $j_*M$ as the $F$-invariants of $\sM$. If $X$ is projective, then $\sM$ is semi-stable of degree $0$. This enables us to find an ample curve which does not depend on $M$ which contains the necessary numerical information.
 If $X$ is not projective, in order for our construction of $\sM$ to be helpful, we need that the degree of $\sM$, which is not positive by construction, is bounded below by a quantity which depends on $M$ only via its rank $r$. 
We deduce this from Abhyankar's lemma. 

For Theorem~\ref{thm:proj}, the inclusion $H^2(\pi^t_1(X), M) \to H^2(X, M)$ above implies $H^2(\pi^t_1(X), M) =0$ by Artin's affine theorem if $\F$ has characteristic not equal to $p$ and by \cite[Theorem 2]{Sha47} if $\F=\F_p$.

\medskip

{\it Acknowledgement}: After posting a first draft of our article, we received various comments, notably on references. We warmly thank Mikhail Borovoi, Luc Illusie and Fabrice Orgogozo for their kind help.

\section{The tame fundamental group of a smooth affine curve}

\begin{prop} \label{FreeQuotientsOfTameFundamentalGroup}

Let $Y$ be a smooth connected affine curve over an algebraically closed field $k$ of characteristic $p >0$, and let $\ell$ be a prime number. 
Then the maximal pro-$\ell$ quotient $\pi^t_1(Y)^{(\ell)}$ of the tame fundamental group $\pi_1^t(Y)$  of $Y$ is a free pro-$\ell$ group.

\end{prop}

\begin{proof}
 
Choose a smooth compactification $Y\hookrightarrow \bar Y$. 
In case $\ell \neq p$, the proposition follows from \cite[Theorem 10.1.6]{NSW13}.
In general, the continuous homomorphism
$\pi_1^t(Y)\to \pi_1(\bar Y)$ is surjective as if $\bar Z\to \bar Y$ corresponds to a finite quotient $H$ of $\pi_1(\bar Y)$, and admits an intermediate finite cover $\bar Z|_Y\to X\to Y$ in restriction to $Y$, then the normalization $\bar X$ of $X$ in the function field of $\bar Z $ is unramified along $\bar Y\setminus Y$.

We assume next that $\ell = p$. 
Since the tame local inertia at a point in $\bar Y\setminus Y$ is a pro-$p'$-group, 
every finite \'{e}tale tame Galois cover of $Y$ of degree a power of $p$  is unramified along $\bar Y\setminus Y$.
We conclude that  the surjection $\pi_1^t(Y)\to \pi_1(\bar Y)$ induces an isomorphism on
$\pi^t_1(Y)^{(p)} \xrightarrow{\cong} \pi_1(\bar{Y})^{(p)}$.
To finish the proof, we apply \cite[Theorem 2]{Sha47} which says that $\pi_1(\bar Y)^{(p)}$ is a free pro-$p$ group. 
\end{proof}

\begin{prop} \label{ProEllFreenessCor}

Let $X$ be a smooth connected affine curve over an algebraically closed field $k$ of characteristic $p>0$, and let $\ell$ be a prime number. 
For every open subgroup $U$ of $\pi_1^t(X)$, the group $U^{(\ell)}$ is free pro-$\ell$. 

\end{prop}

\begin{proof}
Let $X\hookrightarrow \bar X$ be a smooth projective compactification. 
An open subgroup $U \subset \pi_1^t(X)$ corresponds to a finite \'{e}tale cover $Y$ of $X$, tamely ramified along $\bar X\setminus X$.  The following claim reduces Proposition~ \ref{ProEllFreenessCor} to 
Proposition~\ref{FreeQuotientsOfTameFundamentalGroup}.
\end{proof}

\begin{claim}
$U \cong \pi_1^t(Y)$.

\end{claim}
\begin{proof}
Since  $U$ is an open subgroup of $\pi_1^t(X)$, we have 
\begin{equation} \label{OpenSubgroupBasicCofinality}
U \cong \varprojlim_{\substack{N \leq U, \  N \lhd_o \pi_1^t(X)}} U/N,  \notag
\end{equation}
as  the open subgroups $N$ of $U$ which are normal subgroups of $\pi_1^t (X)$  form a cofinal family of open normal subgroups of $U$, namely they intersect trivially. On the other hand the finite \'etale Galois cover 
$Z\to Y$  corresponding to $V\leq U$ is tame if and only if the composite $Z\to Y\to X$, which corresponds to $V\le \pi_1^t(X)$ is. Indeed 
 the corresponding statement is true for extensions of discretely valued rings, that is, if $R\hookrightarrow S$ is tame, then 
$R\hookrightarrow S\hookrightarrow T$ is tame if and only if $S\hookrightarrow T$ is tame. 
Thus $U=\pi_1^t(Y)$.

\end{proof}

\begin{proof}[Proof of Theorem~\ref{thm:proj}]
We need to show that the cohomological dimension of $\pi_1^t(X)$ is at most $1$.
By definition this amounts to showing that, for every prime number $\ell$, the $\ell$-cohomological dimension of $\pi_1^t(X)$ is at most $1$.
By \cite[Corollary 3.3.6]{NSW13}, this is equivalent to the cohomological dimension of an $\ell$-Sylow subgroup $\mathcal{S}_\ell$ of $\pi_1^t(X)$ being at most $ 1$. 
Recall from \cite[page 209]{NSW13} that
\begin{equation}
\mathcal{S}_\ell = \varprojlim_{\mathcal{S}_{\ell} \leq U \leq_o  \pi_1^t(X)} U^{(\ell)} \notag
\end{equation}
which in view of Proposition \ref{ProEllFreenessCor} is an inverse limit of free pro-$\ell$ groups.
We conclude from \cite[Theorem 3.5.15]{RZ10} that $\mathcal{S}_\ell$ is a free pro-$\ell$ group.
In other words, the cohomological dimension of $\mathcal{S}_\ell$ is at most $1$. 
\end{proof}

Let $X$ be as in Theorem~\ref{thm:proj}. Fix a finite \'etale cover $h: X'\to X$ with  $X'$ connected. Following Drinfeld~\cite[Theorem~2.5 (ii)]{Dri12} we pose the following definition.

\begin{defn} \label{defn:drinfeld} 
\begin{itemize}
\item[1)]
 A finite \'etale cover $Y\to X$ has  ramification bounded by $h$ if $Y\times_X X'\to X$ is tame.
 \item[2)]
We say that $h$ is saturated if  $\sK={\rm Ker}(\pi_1(X')\to \pi_1^t(X'))$ is normal in $\pi_1(X)$. 
 \end{itemize}
 If $h$ is saturated we set $\pi_1^h(X)=\pi_1(X)/\sK$.
 \end{defn} 
 Of course any $h$ is dominated by $h': X''\to X'\xrightarrow{h} X$ such that $X''$ is connected, $h'$ is finite \'etale Galois and saturated. 
 With this definition, we see that a finite Galois \'etale cover $Y\to X$ with Galois group $H$ has ramification bounded by $h$ if and only if the quotient $\pi_1(X)\to H$ factors through $\pi_1^h(X)\to H$.  The quotient $\pi_1(X)\to \pi_1^t(X)$ factors through $\pi_1^h(X)$ and is an isomorphism if $h$ is. 

\begin{thm}[Corollary of Theorem~\ref{thm:proj}] \label{thm:Drinfeld}
We assume that $h$ is saturated. Then the  group $\pi_1^h(X)$ is finitely presented. 

\end{thm}

We shall make constant use of the following elementary bounds throughout the article.
\begin{claim}  \label{claim:elem} 
For any profinite group $\pi$, let $M$ be a linear representation of $\pi$ of rank $r$ defined over the prime field $\F$. 
Then 
\begin{itemize}
\item[(i)] $ {\rm dim}_{\F} H^m(\pi, M) \leq |\pi|^m \cdot r \ $ for all $m\in \N$ if $\pi$ is finite;
\item[(ii)]  $ {\rm dim}_{\F} H^1(\pi, M) \leq \delta \cdot r  \ $ if $\pi$ is generated by $\delta$ elements;
\item[(iii)]   $ {\rm dim}_{\F} H^0(\pi, M) \leq r$.
\end{itemize}
\end{claim}
\begin{proof}
(i) comes from the description of $H^m(\pi, M)$ with cocycles viewed as maps $\varphi: \pi^m\to M$, (ii) as well together continuity of $\varphi$ and   the cocycle condition $\varphi(a\cdot b)=a\cdot \varphi(b) + \varphi(a)$, and (iii) follows from the identification of cohomology in degree $0$ with the fixed points of $\pi$ in $M$.
\end{proof}

The following claim will give us practice in applying Theorem~\ref{thm:lub} and the accompanying cohomological machinery. 
Another approach would be to follow the arguments establishing the analogous facts in combinatorial (discrete) group theory.

\begin{claim}  \label{claim:lub}
\begin{itemize}
\item[(i)] 
If $1\to S\to G\to Q \to 1$ is an extension of profinite groups where $S$ and $Q$ are finitely presented,  then  $G$ is finitely presented. 
\item[(ii)] If $U\subset G$ is a finitely presented open subgroup of a profinite group $G$, then $G$ is finitely presented.

\end{itemize}
\end{claim}
\begin{proof} We prove (i). 
A union of a generating set for $S$ with a lift of a generating set for $Q$ is a generating set for $G$, so $G$ is finitely generated.
We can therefore apply Lubotzky's Theorem~\ref{thm:lub}. 
Let $M$ be a finite-dimensional representation of $G$ of rank $r$ defined over a prime field $\F$. 
We use the spectal sequence
\ga{}{ E_2^{a,b}= H^a(Q, H^b(S, M)) \Rightarrow   H^{a+b}( G, M). \notag}
One uses Claim~\ref{claim:elem} for the groups $H^b(S, M)$ with $b=0,1,2$ then 
again {\it loc.cit.} for $E_2^{a,b}$ with $a=0,1,2$. This finishes the proof of (i). 

We prove (ii).  Clearly $G$ is generated by the generators of $U$ and lifts of the finite set $G/U$. For $M$ a rank $r$ representation of $G$ defined over $\F$, note that $N=({\rm Cor}\circ {\rm Res} (M))/M$ is a representation of $G$ of rank $(e-1)r$ where $e=[G:U]$. We apply 
Claim~\ref{claim:elem} (ii) to $H^1(G, N)$ and the assumption to $H^2(U, M)$  in the exact sequence $H^1(G, N)\to H^2(G, M)\to H^2(U, M)$ to finish the proof of (ii).

\end{proof}

\begin{proof}[Proof of Theorem~\ref{thm:Drinfeld}]

By definition, $\pi_1(X')\subset \pi_1(X)$ is open, thus $\pi_1^t(X')= \pi_1(X')/\sK\to \pi_1(X)/\sK$ is open as well. 
We apply Theorem~\ref{thm:proj} to $\pi^t_1(X')$, which together with Claim~\ref{claim:lub} (ii)   concludes the proof.

\end{proof}

\section{ The  prime to $p$ part of Theorem~\ref{thm:main}} \label{sec:pprime}

The aim of this section is to prove Proposition~\ref{prop:tamel} and  to deduce from it the finite presentation of $\pi_1(X)^{(p')}$.

 Let $X$ be a smooth quasi-projective variety defined over an algebraically closed field $k$ of characteristic $p>0$. The choice of a geometric point $x\in X$ enables one to define the tame fundamental group $\pi^t_1(X,x)$  as the automorphism of the fiber functor to Sets  defined by $x$ from the finite \'etale covers which are tame on all curves, 
 \cite[Sec.~4, Definitions, Theorem.~4.4]{KS10}.  In the case $X$ has a smooth projective compactification $j: X\hookrightarrow \bar X$ such that $\bar X\setminus X$ is a  normal crossings divisor, this  is the same definition as \cite[Section~2.2]{GM71}.
  As the isomorphism class of this group does not depend on the choice of $x$, we omit $x$ in the notation. The restriction functor from the category of tame finite \'etale covers of $X$ to the finite \'etale covers 
  defines a surjection $\pi_1(X)\surj \pi_1^t(X)$. We recall the following fact.
  \begin{prop} \label{prop:fg}
   Let $X$ be a smooth quasi-projective variety defined over an algebraically closed field $k$ of characteristic $p>0$. Then its tame fundamental group $\pi_1^t(X)$ and its prime-to-$p$ fundamental group $\pi_1(X)^{(p')}$ are topologically finitely generated.
  
  \end{prop}
  
  \begin{proof}

Let $X \hookrightarrow \bar{X}$ be a normal projective compactification of $X$, with boundary $D=\bar{X}\setminus X$. Let $\bar{C}\subset \bar{X}$ be a complete intersection of dimension $1$ of ample divisors in good position with respect to $D$. Set $C=\bar{C}\cap X$. If $h: Y\to X$ is a connected  \'etale cover,  let us denote by $\bar h: \bar Y\to \bar X$ the normalization of $\bar X$ in the field of functions $k(Y)$. Then $\bar h|_{\bar C}: \bar{B}:=\bar{f}^{-1}(\bar{C})$ is connected. On the other hand, if in addition $h$ is tame, as  $\bar C$  is transversal to the smooth locus of $D$,  $\bar h|_{\bar C}$ is locally Kummer thus $\bar B$ is smooth.  Thus $B=Y\cap \bar B $ is connected as well. 
We then apply Theorem \ref{thm:proj}, in particular its consequence that $\pi_1^t(C)$ is finitely presented.  This concludes the proof for $\pi^t_1(X)$, thus for its quotient $\pi_1(X)^{(p')}$ as well. 
    
\end{proof}
 
We simplify the notation and set $\pi^t=\pi_1^t(X)$ when there is no confusion.  
We  write an exact sequence of profinite groups
\ga{4}{1\to K\to \pi \to \pi^t\to1} 
defining $K$.
\begin{lem} \label{lem:vanK} 
For a prime finite field $\F$ of characteristic different from $p$ we have
$H^1(K, \F)=0$.
\end{lem}
\begin{proof}
The proof is inspired by \cite{GM71},  Proposition~5.2.4 a).
Let $\tilde X \to X$ be the universal cover with base point $x$ and Galois group $\pi$ and $\tilde X  \to X^t\to X$ be the intermediate 'universal tame cover' with Galois group $\pi^t$.  Then $K=\pi_1(X^t,x)$. By definition $H^1(\tilde X, \F)=0$ thus the 
 Lyndon-Hochshild-Serre spectral sequence applied to $\tilde X\to X^t$ 
 implies
$
H^1(K, \F) =H^1(X^t, \F)$. 

On the other hand, $H^1(X^t, \F)=\varinjlim_\alpha H^1(X_\alpha, \F)$ for all the intermediate covers $X^t\to X_\alpha\to X$ with $X_\alpha\to X$ finite.  Let $a\in H^1(K, \F)$, so there is $X_\alpha\to X$ as above so that $a$ is the image of a class 
$a_\alpha\in H^1(X_\alpha, \F) =H^1(\pi^\alpha, 
\F) = H^1(\pi^{\alpha,t}, 
\F)  $ where $\pi^\alpha=\pi_1(X_\alpha)
\subset  \pi$, and $\pi^{\alpha,t }$ is its tame quotient. The latter factorization  follows from the fact that the pro-$\ell$-completion of $\pi^\alpha$  is a quotient of $\pi^{t,\alpha}$. This defines an $\F$-torsor based at $x$
\ga{}{ X_{a(\alpha)} =\tilde X \times_{\pi^\alpha}\F = X^t\times_{  \pi^{\alpha,t} } \F \to X_\alpha \notag}
which is tame, thus the composite 
$ X_{a(\alpha)}\to X_a\to X $
is dominated by $X^t$. Finally by construction the image of $a_\alpha$ in $H^1(X_{a(\alpha)}, \F)$ vanishes. This finishes the proof.

\end{proof}

  \begin{cor} \label{cor:inj}
Let $M$ be a  representation of $\pi^t$ on a finite dimensional $\F$-vector space, where the characteristic of $\F$ is not $p$. Then the maps
\ga{}{ H^2(\pi^t, M) \to H^2(\pi, M)  \to H^2(X, M) \notag}
are injective.

\end{cor}
\begin{proof} 
Let $\tilde X\to X$ be the universal cover with base point $x$. Then by definition
$H^1(\tilde X, \F)=0$. By the Lyndon-Hochshild-Serre spectral sequence we conclude
\ga{22}{ H^i(\pi, M)=H^i(X, M) \ 0\le i\le 1, \ H^2(\pi, M) \hookrightarrow H^2(X, M)}
regardless of the characteristic of $\F$. 
We have the exact sequence 
\ga{}{  (H^1(K, \F) \otimes M)^{\pi^t} \to H^2(\pi^t, M) \to H^2(\pi, M).\notag}
We apply Lemma~\ref{lem:vanK}.
\end{proof}

\begin{prop} \label{prop:tamel}
Let $X$ be a smooth quasi-projective variety defined over an algebraically closed field $k$ of characteristic $p>0$. There is a constant $C\in \R_{>0}$ such that for any  $r$-dimensional linear representation $M$ of $\pi^t$ over a prime field $\F$ of characteristic prime to  $p$ one has 
\ga{}{ {\rm dim}_{\F} H^2(\pi^t, M)\le C\cdot r.\notag}

\end{prop}

\begin{proof}
By the prime to $\ell$-alteration theorem~\cite[Thm.~1.2.5]{Tem17} by de Jong-Gabber-Temkin, there is an  alteration $h: Y\to X$ of degree $\delta$ a power of $p$, 
with a smooth projective  compactification $Y\hookrightarrow \bar Y$ such that $\bar Y\setminus Y$ is a strict normal crossings divisors.

Let $d$ be the dimension of $X$. For $x\in H^2(X, M), y\in 
H^{2d-2}_c(X, M^\vee)$, one has ${\rm Tr} (h^*x\cup h^*y)=\delta (x\cup y) \in H^{2d}_c(X, \F)$ where the trace map ${\rm Tr}:  H^{2d}_c(Y, \F)\to H^{2d}_c(X,\F)$ is an isomorphism. (We do not write Tate twists as $k$ is algebraically closed). If $h^*x=0$, then $h^*x\cup h^*y=0$; as $\delta$ is prime to the characteristic of $\F$, this implies $x\cup y=0$.  Thus if $h^*x=0$ then $x\cup H^{2d-2}_c(X,M^\vee)=0$, which by Poincar\'e duality \cite[Expos\'e~XVIII, Th\'eor\`eme~3.2.5]{SGA4}  implies that $x=0$.  We conclude that $h^*: H^2(X, M)\to H^2(Y, h^*M)$ is injective, thus by Corollary~\ref{cor:inj} that $h^*: H^2(\pi^t, M)\to H^2(\pi_1^t(Y), h^*M)$ is injective as well. In other words, we may assume that  $X$ has a smooth projective  compactification $X\hookrightarrow \bar X$ such that $\bar X\setminus X$ is a strict normal crossings divisors.

We apply \cite[Theorem~1.1 b)]{EK16} and \cite[X, p.90]{SGA2} to reduce to the case where $d=2$. Let $\bar D$ be smooth ample curve on $\bar X$ in good position with respect to $\bar X\setminus X$. Set $D=X\cap \bar D$. 
By purity \cite[XVI, Th\'eor\`eme.~3.7]{SGA4},  the Gysin morphism
$H^0(D, M) \to H^2_D(X, M)$ is an isomorphism, thus the sequence
$H^0(D, M)\to H^2(X, M)\to H^2(X\setminus D, M)$ is exact. 
As ${\rm dim}_{\F} H^0(D, M) \le r$ we may assume that $X$ is affine and we compute $H^2(X,M)$. 
By Deligne's theorem~\cite[Th\'eor\`eme~2.1]{Ill81}  and tameness of $M$, one has 
$\chi(X, M)=r\cdot \chi(X, \F)$. 
By Artin vanishing theorem \cite[Corollaire~3.5]{SGA4} we have $H^3(X,M)=H^4(X,M)=0$. Thus 
\[{\rm dim}_{\F} H^2(X,M) \le {\rm dim}_{\F}H^1(X, M) + r\cdot \chi(X,\F).\] 
We apply Claim~\ref{claim:elem}  and \eqref{22} to $H^1(X, M)$. 

It remains to bound $|\chi(X, \F)|$ independently of the choice of $\F$. This is a general argument for a normal crossings divisor compactification $(\bar X,D)$.  Write $D=\cup_i D_i=\bar X\setminus X$, where the $D_i$ are the irreducible components. 
For $I=\{i_1,\ldots, i_s\}$, set $D_I=\cap_{i\in I} D_i$. Then 
\ga{}{ \chi(X, \F)=\chi(\bar X, \F) +\sum_{|I|=a} (-1)^a \chi( D_I, \F).\notag} 

On the other hand, the integers  $|\chi(\bar X, \F)|, |\chi(D_I, \F)|$ are seen to be independent of the choice of $\F$, as follows. The cohomology groups $H^i(-, \Z_\ell)$ are modules of finite type over $\Z_\ell$.  
If $V$ is a finite type $\Z_\ell$-module, so $V=L\oplus T$ where $L$ is  free and $T=\oplus_1^t \Z_\ell/\ell^{n_i}, n_i\ge 1$ is torsion, one has ${\rm dim}_{\F_\ell} (V\otimes_{\Z_\ell} \F_\ell)={\rm dim}_{\Q_\ell} (V\otimes_{\Z_\ell} \Q_\ell )+t$ while ${\rm dim}_{\F_\ell} 
{\rm Ker} (\cdot \ell: V\to V)=t$.  This together with the (universal coefficient) exact sequences (for $i\geq 0$) 
\ml{}{ 0\to H^i(-,\Z_\ell)\otimes_{\Z_\ell} \F_\ell \to H^i(-,\F_\ell) \to\\ {\rm Ker} (\cdot \ell: (H^{i+1}(-,\Z_\ell) \to H^{i+1}(-,\Z_\ell))) \to 0. \notag}
shows  $\chi(-, \F_\ell)=\chi(-, \Q_\ell)$ and reduces the problem to the independence of $\ell$ of  $|\chi(\bar X, \Q_\ell)|, |\chi(D_I, \Q_\ell|.$  This last  point can be seen in at least two ways.  We can argue using Deligne's purity of cohomology \cite[Th\'eor\`eme~3.3.1]{Del80} out of which we conclude that even the single dimensions  ${\rm dim}_{ \Q_\ell} H^i(\bar X, \Q_\ell)$  and ${\rm dim}_{\F} H^i(D_I, \Q_\ell)$ do  not depend on $\ell\neq p$. Or if $k=\bar \F_p$ we argue that $\chi$ is  the difference of the degree of the denominator and the degree of the numerator of the  zeta function. If $k\neq \bar \F_p$ one reduces to this case by specialization. Either way this finishes the proof.

\end{proof}

We now deduce the finite presentation of $\pi_1(X)^{(p')}$ (see the introduction). 
\begin{cor} \label{cor:primetop}
Let $X$ be a smooth connected quasi-projective variety defined over an algebraically closed field $k$ of characteristic $p>0$.  Then  $\pi_1(X)^{(p')}$ is finitely presented.

\end{cor}
\begin{proof}
We  set $\pi^{(p')}=\pi_1^{(p')}(X,x)$ to simplify the notation, which by Proposition~\ref{prop:fg} is finitely generated. We  apply Theorem~\ref{thm:lub} to   an $r$-dimensional linear representation $M$ of  $\pi^{(p')}$. 
As the order $|\pi^{(p')}|$ is prime to $p$ as a supernatural number,  one has $H^2(\pi^{(p')}, M)=0$ 
from \cite[\S3. Proposition~11,~Corollary 2]{Ser97}  when $\F=\F_p$.  We now assume that the characteristic of $\F$ is prime to $p$. We define the exact sequence
\ga{}{1\to K^{(p')}\to \pi^t\to \pi^{(p')}\to 1 \notag}
yielding the exact sequence of cohomology
\ga{5}{ (H^1(K^{(p')}, \F)\otimes_{\F} M)^{\pi^{(p')}} \to H^2(\pi^{(p')},  M) \to H^2(\pi^t, M).} 

By \cite[Lemma~3.4.1(d)]{RZ10}, applied with $K=K^{(p')}$, and the formation $\mathcal{C}$ of finite groups of order prime to $p$, one has 
\ga{6}{ H^1(K^{(p')}, \F)=0.}
Indeed, let $\phi \colon K^{(p')} \to \F_\ell$ be a continuous group homomorphism.
By \cite[ Lemma~1.2.5(c)]{FJ08}, $\phi$ extends to $\Phi: U\to \F$ where 
   $ K^{(p')} \subset U \lhd \pi^t$ and $U$ is open normal  in $\pi^t$. Set $L={\rm Ker} ( \Phi) $. 
   As $U/L \subset \F$, $U/gLg^{-1} \subset \F$  for any $ g\in \pi^t$ as well.  Set $R=\cap_{g\in \pi^t} gLg^{-1}\subset \pi^t$, which is normal in $\pi^t$. Thus
   $U/R \subset \oplus_{\rm finite} \F$  and $\pi^t/R$ is an extension of $\pi^t/U$ (a finite quotient of $\pi^{(p')}$)  by the subgroup of  $\oplus_{\rm finite} \F$, which has order prime to $p$.  In particular, $\pi^t\twoheadrightarrow\pi^t/R$ factors through $\pi^{(p')}$. Thus $L\supset R\supset K^{(p')}$, which shows $\phi=0$. 
   
   Thus \eqref{6} implies that the  term  on the left in \eqref{5} vanishes. We now apply 
   Proposition~\ref{prop:tamel} to the term on the right to finish the proof.

\end{proof}

\section{ The tame fundamental group of a smooth quasi-projective variety admitting a good compactification}

The aim of this section is to prove Theorem~\ref{thm:main}.   We set $\pi^t=\pi^t_1(X)$ to simplify the notation. 
 Given Proposition~\ref{prop:tamel} which is true even if one does not have a good compactification, it remains  to bound 
$H^2(\pi^t, M)$  linearly in $r$  when $M$ is defined over $\F=\F_p$.   By  $\underline{M}$ we denote the local system associated to $M$. 

\begin{prop} \label{prop:numtame}
Let $X$ be as in Theorem~\ref{thm:main}, $M$ be a linear representation of $\pi_1^t(X)$ defined over $\F_p$.  Then there is an injective $\F_p$-linear map $H^2(\pi_1^t(X), M)\to H^2(\bar X, j_*\underline{M})$.

\end{prop}

\begin{proof}
For each open normal  subgroup $U \lhd_o \pi^t$
we define the fibre square
\ga{}{\xymatrix{ \ar[d]_{h_U} X_U \ar[r]^{j_U} & \bar X_U \ar[d]^{\bar h_U}\\
X\ar[r]^j& \bar X
} \notag}
where $h_U$ is the tame Galois cover defined by $\pi^t\to G_U=\pi^t/U$ and $\bar X_U$ is the normalization of $\bar X$ in the field of functions $k(X_U)$ of $X_U$. 
 Then $j_{U*} \underline{M}^U$ is a $G_U$-equivariant constructible sheaf on $\bar X_U$.  By \cite[Theorem]{Suw79}, Grothendieck's spectral sequence \cite[Th\'eor\`eme~5.2.1]{Gro57} 
\ga{}{E_2^{ab}=H^a(G_U, H^b(\bar X_U,j_{U*} \underline{M}^U)) \Rightarrow 
H^{a+b}(\bar X_U, G_U, j_{U*} \underline{M}^U).\notag
 }
exists where on the right stands equivariant cohomology.  By definition for $U\subset {\rm Ker}( \pi^t\to GL(M))$,  $\underline{M}^U= h_U^*\underline{M}$ thus
\ga{a}{ \bar h^G_{U*} j_{U*} \underline{M}^U=j_* \underline{M}, }
 where $\bar h^G_{U*}$ is the $G$-invariant direct image functor.

As $j$ is a normal crossings compactification, $\bar h_U$ is numerically tame in the sense of \cite[Section~5]{KS10}, see \cite[Theorem~5.4 (a)]{KS10}.  
Thus by \cite[Corollaire p.204]{Gro57} one has 
\ga{b}{ R^{>0}   \bar h^G_{U*} j_{U*} \underline{M}^U=0 , \\
H^{n}(\bar X_U, G_U, j_{U*} \underline{M}^U)=H^n(\bar X, j_*\underline{M}). \notag
}
We conclude that the spectral sequence $E_2^{ab}$ yields for each $U$ a short exact sequence 
\ga{seq}{ H^0(G_U, H^1(\bar X_U, j_{U*} \underline{M}^U)) \to H^2(G_U, H^0(\bar X_U,j_{U*} \underline{M}^U)) \to H^2(\bar X, j_*\underline{M})\notag}
which is compatible with the restriction maps as we make $U$ smaller.
For $U$ contained in ${\rm Ker} \big(\pi_1^t(X)\to GL(M)\big)$, we have 
$H^1(\bar X_U, j_{U*} \underline{M}^U)=H^1(\bar X_U, \F_p)\otimes_{\F_p} M$ where on the right $\otimes_{\F}M$ means the $\F_p$-vector space as a $G_U$-representation.  There is an open (normal if we wish) subgroup $\bar V \lhd_o \pi_1(\bar X_U)$ such that the induced map $H^1(\bar X_U, \F_p) \to 
H^1((\bar X_U)_{\bar V}, \F_p)$ is equal to $0$. We take  $V\lhd_o \pi^t$ open normal  and contained in the inverse image of $\bar V$ under the surjection $\pi_1^t(X_U)\to \pi_1(\bar X_U)$.  Then the map 
$H^1(\bar X_U, j_{U*} \underline{M}^U)\to H^1(\bar X_V, j_{V*} \underline{M}^V)$ is equal to zero. 
 We conclude that 
\ga{}{ \varinjlim_U H^2(G_U, M^U)\to H^2(\bar X, j_*\underline{M}) \notag}
is injective. 
By \cite[Proposition~1.2.5] {NSW13} we have 
\ga{2}{ H^2(\pi^t, M)=\varinjlim_U H^2(G_U, M^U) \notag}
where $U$ ranges over all open normal subgroups of $\pi^t$ and $G_U=\pi^t/U$. This finishes the proof.

\end{proof}

\begin{proof}[Proof of Theorem~\ref{thm:main}] 
As $j$ is a normal crossings compactification, we apply \cite[Theorem~1.1 b)]{EK16} and \cite[X, p.90]{SGA2} to reduce to the case of dimension equal to $2$. 
We fix a linear representation $M$ of rank $r$ of $\pi^t$ over $\F_p$.
For $U={\rm Ker}(\pi^t\to GL(M))$ we define $h_U, \bar h_U$ as in the proof of 
Proposition~\ref{prop:numtame}. Then $h_U^*M$ is a trivial representation of $U$. We set $L= H^0(X_U,  h_U^*\underline{M})$ for its $r$-dimensional vector space of global sections. We define $\sL=L\otimes_{\F_p} \sO_{\bar X_U}$, a locally free sheaf of rank $r$  which  admits an action of the Frobenius $F_U$ on $\bar X_U$ with invariant sections $L$. The Artin-Schreier sequence on $(\bar X_U)_{\rm \acute{e}t}$ 
\ga{}{0\to \underline{L} \to \sL \xrightarrow{1\otimes F_U  -1} \sL\to 0 \notag}
is  an exact sequence of $G_U$-equivariant sheaves.  By definition, we have   $\underline{L}=j_{U*} h_U^*\underline{M}$. 
Applying \eqref{a} and \eqref{b}  we obtain the induced exact sequence 
\ga{}{0\to j_*\underline{M}\to \sM\xrightarrow{\psi} \sM\to 0\notag}
on $\bar X_{\rm \acute{e}t}$ 
where $\sM= \bar h^{G_U}_{U*}\sL$, $\psi= \bar h^{G_U}_{U*}( 1\otimes F_U-1).$
On the other hand,
 the map $\psi$ is of the form $\psi_F-1$, where $\psi_F=\bar h^{G_U}_{U*}(1\otimes F_U)$ is a $p$-linear endomorphism of $\sM$. Since
$H^1(\bar X,\sM)$ is finite dimensional over $k$, we have that 
\[\psi_F-1:H^1(\bar X,\sM)\to H^1(\bar X,\sM)\]
is surjective, by \cite[XXII, Proposition~1.2]{SGA7.2}. Similarly the finite dimensionality of $H^2(\bar X,\sM)$ and \cite[XXII, (1.0.10) and Proposition~1.1]{SGA7.2}
imply that
\ga{}{H^2(\bar X, j_*\underline{M})\otimes_{\F_p}k\subset H^2(\bar X, \sM)\notag}
and it suffices to bound the dimension over $k$ of the right hand side.

 As $\sL$ is locally free (in fact trivial) of rank $r$,  it is reflexive, thus $\sM$ is reflexive and of generic rank $r$.  As $\bar X$ is smooth and has dimension equal to $2$, $\sM$ is locally free as well, and of rank $r$. By definition
 \ga{}{ \bar h_U^*\sM\subset \sL.\notag}
 On the other hand, one has
 \ga{}{ \sL\otimes_{\sO_{\bar X_U}} \bar h_U^*\sO_{\bar X}(-D)\subset \bar h_U^*\sM\notag}
 where $D$ is the reduced  normal crossings divisor $\bar X
 \setminus X$. Indeed, as both sheaves are locally free, $\bar X_U$ is normal, and $\bar h_U^*\sM=\sL$ on $X_U$, 
  it is enough to check  the inclusion in restriction to $\Delta \to \bar X_U$ where $\Delta={\rm Spec}(R)$, $R$ is a discrete valued ring with residue field $k$, so that $\Delta$ is in good position with respect to $D$. Then  $\bar X_U\times_{\bar X} \Delta\to \Delta$  is a disjoint union of Kummer covers $
T \  \to \Delta,$  where $T={\rm Spec}(S), \ S=R[t]/(t^n-x\cdot u)$,  $u\in R^\times$, $x$ is the uniformizer of $R$, and  $(n,p)=1$, which are permuted by the quotient of $G_U$ by the stabilizer of $T$.    Thus $\sM|_{\Delta}$, as an $R$-module, is of the shape $\oplus_1^r t^{a_i} R$, where $0\le a_i< n$, so $\sM|_{\Delta} \otimes_R  S=\oplus_1^r t^{a_i} S \supset \oplus_1^t t^n S= \sL\otimes_{\sO_{\bar X_U}} \bar h_U^*\sO_{\bar X}(-D)\otimes_{\sO_{\bar X_U}} S.  $ So we conclude.

By  Serre duality $ H^2(X, \sM)=  H^0(\bar X, \sM^\vee \otimes_{\sO_{\bar X}}  \omega_{\bar X})^\vee$ where $\omega_{\bar X}$ is the dualizing sheaf of $\bar X$. 
Let $\sA$ be a very ample line bundle on $\bar X$ such that 
 $\sA\otimes_{\sO_{\bar X}}\omega_{\bar X}(D)=\sB$ is very ample as well. 
  Let $C$  and $C'$ be  smooth curves on $\bar X$ such that 
  $\sB=\sO_{\bar X}(C)=\sO_{\bar X}(C')$, and $C, C'$ intersect transversally and are both in good position with respect to $D$. Then  
  \ml{}{
 \bar h_U^*(\sM^\vee \otimes_{\sO_{\bar X}} \omega_{\bar X}(-C))\subset \\
\sL^{\vee}\otimes_{\sO_{\bar X_U}} 
\bar h_U^*(\omega_{\bar X}(D-C))=\sL^{\vee}\otimes_{\sO_{\bar X_U}}\bar h_U^*\sA^{-1}. \notag}
Since $\sL$ is a trivial vector bundle, and by choice, $\omega_{\bar X}^{\vee}(C-D)=\sA$ is very ample on $\bar{X}$, we deduce that 
\ml{}{ H^0(\bar X_U,   \bar h_U^*(\sM^\vee \otimes_{\sO_{\bar X_U}} \omega_{\bar X}(-C)))=\\
H^0(\bar X_U\times_{\bar X}  C,   \bar h_U^*(\sM^\vee \otimes_{\sO_{\bar X_U}} \omega_{\bar X}(-C')))
=
0\notag}
which implies
\ga{}{ H^0(\bar X, \sM^\vee \otimes_{\sO_{\bar X}} \omega_{\bar X}(-C))= H^0(C, \sM^\vee \otimes_{\sO_{\bar X}} \omega_{\bar X}(-C')|_C)=
0. \notag}
So  the restriction maps
\ml{}{ H^0(\bar X, \sM^\vee\otimes_{\sO_{\bar X}} \omega_{\bar X})  \to 
H^0(C, (\sM^\vee\otimes_{\sO_X} \omega_X)|_C) \\ \to H^0(C\cap C', (\sM^\vee\otimes_{\sO_X} \omega_X)|_{C\cap C'}) 
\notag}
 are injective, out of which we conclude 
 \ga{}{{\rm dim}_{k} H^0(\bar X, \sM^\vee\otimes_{\sO_{\bar X}} \omega_{\bar X}) \le r\cdot c_1(\sB)^2.  \notag }
 This finishes the proof.

\end{proof}

\section{Remarks} \label{sec:rmk}

\subsection{Perfect fields} Theorem~\ref{thm:main}  remains true if we relax the assumption on $k$ being algebraically closed under the following assumptions:
\begin{itemize}
\item[(i)] $k$ is perfect;
\item[(ii)] the Galois group $G_k$ of $k$ is finitely presented. 

\end{itemize}
For example, $k$ could be a finite field. 
\begin{proof}
We have the homotopy exact sequence 
\ga{}{1\to \pi_1(X_{\bar k})\to \pi_1(X)\to G_k\to 1\notag}
\cite[Th\'eor\`eme~9.6.1]{SGA1}. On the other hand, tameness is preserved by base change thus the composite  $\pi_1(X_{\bar k}) \to \pi_1(X)\to \pi_1^t(X)$ factors through $\pi_1^t(X_{\bar k}) \to \pi_1^t(X)$. 
\begin{claim} \label{claim:left}
$\pi_1^t(X_{\bar k}) \to \pi_1^t(X)$ is injective. 

\end{claim}
To check this we have  a criterion. 
\begin{claim} \label{claim:crit}
$\pi_1^t(X_{\bar k})\to \pi_1^t(X)$  is injective if and only if for any $ Y \to X_{\bar k}$ finite \'etale tame connected, there is $Z \to X$ finite \'etale tame connected such that one has a factorization
\ga{}{\xymatrix{\ar[d] Y \ar@{.>}[r]^{\exists } & Z \ar[d]\\
X_{\bar k} \ar[r] & X
} \notag}
\end{claim} 
\begin{proof} We have the base points $x\in X_{\bar k} \mapsto x\in X$, so 
\ga{}{ \pi_1^t(X_{\bar k})={\rm Aut}(\omega_x), \omega_x: {\rm Rev}^t(X_{\bar k}) \to {\rm Sets}, (\sigma: Y\to X_{\bar k})\mapsto   \sigma^{-1}(x) \notag \\
\pi_1^t(X_t)={\rm Aut}(\omega_x), \omega_x: {\rm Rev}^t(X)\to {\rm Sets}, (\pi: Z\to X)\mapsto   \pi^{-1}(x).\notag}
Here ${\rm Rev}^t(X_{\bar k}) $ resp. ${\rm Rev}^t(X) $ denotes the category of finite \'etale  tame covers, and Sets the one of sets.
 So $\pi_1^t(X_{\bar k})\to \pi_1^t(X)$ injective is equivalent to the  ${\rm Aut}(\omega_x)$-representation $\sigma^{-1}(x)$ for each $\sigma$ to be a sub- ${\rm Aut}(\omega_x)$-representation  of a 
${\rm Aut}(\omega_x)$ representation $\pi^{-1}(x)$.

\end{proof}
We have fixed a geometric point $x$ which defines the algebraic closure $k\hookrightarrow \bar k\hookrightarrow k(x)$.  As $k$ is perfect, $\bar k$ is the separable closure thus the surjection 
$\pi_1(X)\to G_k$ factors through $\pi_1^t(X)\to G_k$. 
As $X_{\bar k}$ is of finite type, any $\sigma: Y\to X_{\bar k}$  finite \'etale  tame  is defined over a finite extension $k'\supset k$, so say $Y_{k'}\to X_{k'}$, where $k'\supset k$ is \'etale. Thus 
$Y_{k'}\to X_{k'}\to X$ is finite \'etale tame.  This proves Claim~\ref{claim:left}. 
On the other hand, as the image of $\pi_1(X_{\bar k})\to \pi_1(X)$ is normal, and $\pi_1(X_{\bar k})\to \pi^t_1(X_{\bar k})$ is surjective, the image of $\pi^t_1(X_{\bar k})\to \pi^t_1(X)$ is normal as well. 
 Thus the homotopy exact sequence descends to an exact sequence 
\ga{}{1\to \pi^t_1(X_{\bar k})\to \pi^t_1(X)\to G_k\to 1.\notag}
We now apply Claim~\ref{claim:lub}  (i) to finish the proof.

\end{proof}

\subsection{Analogy with Deligne's canonical extension}
The complex 
\ga{}{\sM \xrightarrow{\psi} \sM \notag}
in the proof of Proposition~\ref{prop:numtame} is analog  on $X_\C$ smooth over $\C$ with a normal crossings divisor compactification $j: X_\C\hookrightarrow \bar X_{\C}$ to the logarithmic de Rham complex $\Omega_{X_\C} ^\bullet({\rm log}( \bar X_\C\setminus X_\C)) \otimes \sM_\C$ where $\sM_\C$ is Deligne's canonical extension associated to a  complex local system $\underline{M}_\C$ \cite[Proposition~.5.2, p.~91]{Del70}.
The property that under the numerical tameness assumption  this complex on $\bar X_{\rm \acute{e}t}$ computes $j_*\underline{M}$ is analog to the  analytic property  on the local monodromies along $ \bar X_\C\setminus X_\C$ which force the condition
 $j_*\underline{M}_\C=Rj_*\underline{M}_\C$.

\end{document}